\documentclass[a4paper,11pt]{amsart}
\usepackage[left=2.7cm,right=2.7cm,top=3.5cm,bottom=3cm]{geometry}

\usepackage{amsthm,amssymb,amsmath,amsfonts,mathrsfs,amscd}
\usepackage[latin1]{inputenc}
\usepackage[all]{xy}
\usepackage{latexsym}
\usepackage{longtable}
\usepackage{diagrams}
\usepackage{stmaryrd}

\newfont{\cyr}{wncyr10 scaled 1100}
\newfont{\cyrr}{wncyr9 scaled 1000}

\theoremstyle{plain}
\newtheorem{theorem}{Theorem}[section]
\newtheorem{proposition}[theorem]{Proposition}
\newtheorem{lemma}[theorem]{Lemma}

\theoremstyle{definition}

\newtheorem{definition}[theorem]{Definition}

\theoremstyle{remark}
\newtheorem{remark}[theorem]{Remark}

\newtheorem*{acknowledgements}{Acknowledgements}

\newcommand{\Q}{\mathbb Q}
\newcommand{\N}{\mathbb N}
\newcommand{\Z}{\mathbb Z}

\newcommand{\C}{\mathbb C}

\DeclareMathOperator{\Spec}{Spec}
\DeclareMathOperator{\Pic}{Pic}
\DeclareMathOperator{\End}{End}
\DeclareMathOperator{\Aut}{Aut}
\DeclareMathOperator{\Frob}{Frob}

\DeclareMathOperator{\Hom}{Hom}
\DeclareMathOperator{\Div}{Div}
\DeclareMathOperator{\Gal}{Gal}
\DeclareMathOperator{\GL}{GL}
\DeclareMathOperator{\PGL}{PGL}

\DeclareMathOperator{\M}{M}

\DeclareMathOperator{\Spf}{Spf}

\newcommand{\res}{\mathrm{res}}

\newcommand{\har}{\mathrm{har}}

\usepackage[usenames]{color}
\definecolor{Indigo}{rgb}{0.2,0.1,0.7}
\definecolor{Violet}{rgb}{0.5,0.1,0.7}
\definecolor{White}{rgb}{1,1,1}
\definecolor{Green}{rgb}{0.1,0.9,0.2}

\newcommand{\longepi}{\mbox{$\relbar\joinrel\twoheadrightarrow$}}

\def \mint {\times \hskip -1.1em \int}

\newcommand{\p}{\mathfrak p}

\newcommand{\cO}{{\mathcal O}}
\newcommand{\E}{{\mathcal E}}
\newcommand{\HH}{{\mathcal H}}
\newcommand{\T}{\mathcal T}

\newcommand{\F}{\mathbb F}
\newcommand{\PP}{\mathbb P}

\include{thebibliography}

\begin{document}
\title[]{Exceptional zero formulae for anticyclotomic $p$-adic $L$-functions of elliptic curves in the ramified case}
\author{Matteo Longo and Maria Rosaria Pati}
\maketitle
\begin{abstract}
Iwasawa theory of modular forms over anticyclotomic $\Z_p$-extensions of imaginary quadratic fields has been studied by several authors,
starting from the works of Bertolini-Darmon and Iovita-Spiess, under the crucial assumption that the prime $p$ is \emph{unramified} in $K$. We start in this article the systematic study of anticyclotomic $p$-adic $L$-functions when $p$ is \emph{ramified} in $K$. In particular, when 
$f$ is a weight $2$ modular form attached to an elliptic curve $E/\Q$ having 
multiplicative reduction at $p$, and $p$ is ramified in $K$, 
we show an analogue of the exceptional zeroes phenomenon investigated by 
Bertolini-Darmon in the setting when $p$ is inert in $K$. More precisely,  we consider 
situations in which the  $p$-adic $L$-function $L_p(E/K)$ of $E$ over the anticyclotomic $\Z_p$-extension of $K$ 
does not vanish identically but, by sign reasons, has a zero at certain characters $\chi$ of the Hilbert class field of $K$.  
In this case we show that the value at $\chi$ of the first derivative of $L_p(E/K)$ 
is equal to the formal group logarithm of the specialization at $p$ of a global point on the elliptic curve (actually, this global point is a twisted sum of Heegner points).  
This  
generalizes similar results of Bertolini-Darmon, available when $p$ is inert in $K$ and 
$\chi$ is the trivial character.
\end{abstract}

\section{Introduction}

Cyclotomic Iwasawa theory of elliptic curves goes back to the work of B. Mazur \cite{Mazur} in the seventies, and 
since then its interest has constantly grown; spectacular results, among others, of K. Kato \cite{Kato} and Skinner-Urban \cite{SU} led 
to the proof of the Iwasawa Main Conjecture for a rational elliptic curve, ordinary at $p$, over the cyclotomic $\Z_p$-extension of $\Q$. Parallel to 
the cyclotomic theory, the anticyclotomic Iwasawa theory of elliptic curves has been developed, 
by several authors, among which M. Bertolini, H. Darmon, A. Iovita, M. Spiess, B. Howard, R. Pollack, T. Weston  
(\cite{BD-Heegner}, \cite{BD-CD}, \cite{BD-IMC}, \cite{BD1}, \cite{BDIS}, \cite{How1}, \cite{How2}, \cite{pollack-weston}).

Bertolini and Darmon, in \cite{BD-CD} and \cite{BD1} and later on all the articles dealing with anticyclotomic $p$-adic $L$-functions attached to a pair consisting of a modular form and an imaginary quadratic field $K$, make the crucial assumption that the prime $p$ is unramified in $K$. We start in this article the systematic study of the construction and special values of anticyclotomic $p$-adic $L$-functions attached to a modular form and an imaginary quadratic field $K$ when $p$ is ramified in $K$.

In the setting of Bertolini-Darmon, expecially in \cite{BD-Heegner} and \cite{BD-CD}, one fixes an elliptic curve $E/\Q$ and a quadratic imaginary field $K/\Q$ whose discriminant is prime to the conductor $N$ of $E$; one then factors $N=N^+N^-$ into a product of coprime integers 
in such a way that a prime number divides $N^+$ if and only if it is split in $K$, and consequently a prime number 
divides $N^-$ if and only if it is inert in $K$. Pick a prime number $p$ which does not divide the discriminant of $K$; 
we finally assume that the number of primes dividing $N^-$ is \emph{odd}. 
In this setting, Bertolini-Darmon defined a $p$-adic $L$-functions $L_p(E/K)$
of $E$ over $K$; 
this is an element of the Iwasawa 
algebra $\Z_p\llbracket\tilde G\rrbracket$ 
(where $\tilde G$ is the Galois group of the union of all ring class fields of $K$ of 
$p$-power conductor), characterized by the fact that for all finite order characters $\chi$ of $\tilde G$, the value 
$L_p(E/K)(\chi)=\chi\left(L_p(E/K)\right)$ interpolates special values of the 
complex $L$-function $L(E/K,\chi,s)$ of $E/K$ twisted by $\chi$. In the special case when 
$p$ is inert in $K$ and $p\mid N$, so that $E$ has multiplicative reduction at $p$, 
the main result of \cite{BD-CD}  expresses the value of the first $p$-adic derivative of $L_p(E/K)$ at the trivial character 
as a $p$-adic logarithm of a 
global point in $E(K)$, which is actually a Heegner point of $E$ of conductor $1$ arising from a Shimura 
curve uniformization of $E$. This remarkable result provides a completely 
$p$-adic construction of global points.  

As mentioned before, one of the aims of this paper is to investigate an analogue of the results in \cite{BD-CD} in which the crucial 
assumption made in \emph{loc. cit.} that $p\mid N$ and $p$ is inert in $K$ is replaced by the requirement that 
$p\mid N$ and $p$ is \emph{ramified} in $K$. We obtain results in the spirit of \cite{BD-CD}, expressing the $\chi$-value 
of the derivative of the $p$-adic $L$-function as the $p$-adic logarithm of Heegner points in the $\chi$-eigencomponent of the 
Mordell-Weil group of $E$, for all $\chi$ in a specific 
family of finite order characters of the Galois group of the union of the ring class fields of $p$-power conductor of $K$. Note that the case when $p$ is ramified 
is in some sense richer than the original setting of \cite{BD-CD} when $p$ is inert, since in the latter case 
only the value of the derivative of the $p$-adic $L$-function at the \emph{trivial} character can be studied, 
while in our setting the family of characters giving rise to the $p$-adic construction of global points might be 
strictly larger.  

We now explain in a more precise form the main results of this paper. 
Let $E$ be an elliptic curve over $\Q$ of conductor $N$ and let $p$ be a prime of multiplicative reduction for $E$, so that $p\parallel N$. Fix an imaginary quadratic field $K$ of discriminant $D$. The following assumption are made throughout:
\begin{itemize}
\item $\cO_K^\times=\{\pm 1\}$, i.e. $K\neq\Q(i)$ and $K\neq\Q(\sqrt{-3})$;\
\item $(N,D)=p$, i.e. $p$ is \emph{ramified} in $K$ and no other prime dividing $N$ ramifies in $K$.
\end{itemize}
Write $N$ as $$N=pN^-N^+$$ where the primes dividing $N^-$ are those inert in $K$, and the primes of $N$ dividing $N^+$ are those splitting in $K$. We make the following further assumptions:
\begin{itemize}
\item The integer $N^-$ is squarefree, i.e. $E$ has multiplicative reduction at the primes dividing $N$ which are inert in $K$;\
\item The integer $N^-$ is the product of an odd number of primes.
\end{itemize}
Under the above assumptions, let $K_\infty=\cup_n K_n$ where $K_n$ is the ring class field of $K$ 
of conductor $p^n$; this is the abelian extension of $K$ corresponding to $\Pic(\mathcal{O}_{p^n})$ by class field 
theory, where $\mathcal{O}_{p^n}=\Z+p^n\mathcal{O}_K$. Let $\tilde{G}=\Gal(K_\infty/K)$, which is isomorphic to
a product $\tilde G=T\times G$, where $G\simeq\Z_p$ and $T$ is a finite group. In the first 
part of the paper, following \cite{BDIS}, for an eigenform $\phi$ of even weight $k$ on $\Gamma_0(N)$ we construct an element in $\Z_p\llbracket \tilde G\rrbracket$ which 
is known, thanks to \cite[Theorem 1.4.2]{YZZ}, to interpolate the algebraic part $L^\mathrm{alg}(\phi,K,\chi,k/2)$ of the special value $L(\phi,K,\chi,k/2)$ of the complex 
$L$-function of $\phi$ over $K$ twisted by $\chi$, for all finite order ramified characters $\chi$ of $\tilde G$. Applying  all this to the modular eigenform of weight two attached to $E/\Q$ by modularity, we obtain a $p$-adic $L$-function $L_p(E/K,\chi,s)$, which is an element of $\Z_p\llbracket \tilde G\rrbracket$, interpolating the algebraic part of the special value $L(E/K,\chi,1)$ of the complex $L$-function of $E/K$ twisted by $\chi$.

Let $H_p$ be the maximal subextension of $H$, the Hilbert class field of $K$, in which the 
unique prime $\p$ of $K$ above $p$ splits completely. 
We consider the set $\mathcal{Z}$ of finite order characters $\chi$ of $\tilde{G}$ which 
factors through $H_p$. Fix such a character $\chi$ of $\tilde G$. 
Under our assumption, the elliptic curve $E$ can be uniformized by the Shimura 
curve $X=X_{N^+,pN^-}$ attached to the indefinite quaternion algebra $\mathcal{B}$ of discriminant $pN^-$ and 
the Eichler order $\mathcal{R}$ of level $N^+$; note that the number of primes dividing $N^-$ is odd. Let $P_H$ be a Heegner point of conductor $1$ in $E(H)$ 
arising from this uniformization, and 
construct the point 
\[P_\chi=\sum_{\sigma\in\Gal(H/K)}P_H^\sigma\otimes\chi^{-1}(\sigma)\]
in $(E(H)\otimes_\Z\bar\Q)^\chi$; here for any $\C[\Gal(H/K)]$-module $M$, we denote by
$M^\chi$ its $\chi$-eigencomponent. The main result of this paper is the following: 

\begin{theorem}\label{MainThm}
For all $\chi\in\mathcal{Z}$, 
\[L'_p(E/K)(\chi)=\frac{2}{[H:H_p]}\cdot\left(\log_E(P_\chi)-\log_E(\bar{P}_\chi)\right)\]
where $\log_E$ is the logarithm of the formal group of $E$ and $\bar{P}_\chi$ is the complex conjugate of $P_\chi$. 
\end{theorem}

Of course, the question is when the point $P_\chi$ in the above theorem is non-torsion. 
To discuss this point, let $b_p$ be the $p$-th Fourier coefficient 
of the Theta series attached to $\chi$. Also, let $a_p=+1$ if $E$ has split multiplicative reduction at $p$, 
and $a_p=-1$ otherwise. If 
\begin{equation}\label{parity} 
a_p\cdot b_p=1\end{equation}
then the order of vanishing of $L(E/K,\chi,s)$ at $s=1$ is odd; thus in particular $L(E/K,\chi,1)$ is equal to zero. Under the assumption that $L'(E/K,1)\neq 0$, 
the point $P_\chi$ is non-torsion by \cite[Theorem 1.3.1]{YZZ} and therefore Theorem \ref{MainThm} gives a purely $p$-adic construction of the logarithm of a global point of the elliptic curve, in the 
spirit of \cite[Theorem B]{BD-CD}, As remarked above, note that in our setting the character $\chi$ needs not be trivial.

If \eqref{parity} is not satisfied, the situation looks more mysterious. The order of vanishing of $L(E/K,\chi,s)$ at $s=1$ is even in this case. Call $r$ this order of vanishing. Clearly, if $r\geq 2$, then the point $P_\chi$ is torsion 
by \cite[Theorem 1.3.1]{YZZ}. If $r=0$, that is $L(E/K,\chi,1)\neq 0$, at least assuming $E$ does not have CM, $(E(H)\otimes_\Z\bar\Q)^\chi=0$ by \cite[Theorem A']{Nek}; therefore, 
$P_\chi$ is torsion as well, and $\log_E(P_\chi)$ and $\log_E(\bar{P}_\chi)$ are both zero. 
Therefore, our theorem implies that $L'_p(E/K)(\chi)=0$. It might be interesting in this case to study the arithmetic significance of the second derivative $L''_p(E/K)(\chi)$ of the $p$-adic $L$-function 
valued at $\chi$. 

\begin{remark} 
S. Molina \cite{Molina-preprint} has recently obtained a construction of the $p$-adic $L$-function in the ramified case, following 
the strategy of Spiess (see \cite{Spiess}), and proving several results toward exceptional zero formulae in the more general 
context of totally real number fields; however, 
the author does not cover the case we are interested in. Moreover, the construction of the $p$-adic $L$-function in \cite{Molina-preprint} is different from our.\end{remark}

Finally, let us outline the main differences between the inert \cite{BD-CD} and the ramified case. Leaving  
to the following sections the discussion of the specific technical issues arising in the ramified case 
(for example, $H_p$ might be different from $H$ in our context, 
while $H=H_p$ in the inert case) one of the main differences between the ramified and the inert cases 
is related to the use of the Cerednik-Drinfeld uniformization theorem, which plays an important role in the $p$-adic 
analytic description of the image of Heegner points of conductor $1$ 
in the fiber at $p$ of Shimura curves, in the case when $p$ divides the discriminant of the quaternion algebra. 
Briefly, in the inert case the image of Heegner points of conductor $1$ in $X(\C_p)$ (where $X$ is as above 
the Shimura curve attached to the indefinite quaternion algebra of discriminant $pN^-$ and an Eichler order of level $N^+$)
can be described by means of the Cerednik-Drinfeld theorem and 
correspond to points in the subset $\Q_{p^2}-\Q_p$ of the Drinfeld upper half plane $\mathcal{H}_p=\C_p-\Q_p$, 
where $\Q_{p^2}$ is the quadratic \emph{unramified} extension of $\Q_p$; 
this is due to the fact that the Shimura curve $X/\Q_p$ is a Mumford curve, isomorphic over $\Q_{p^2}$ 
to quotients of the $p$-adic upper half plane by appropriate arithmetic subgroups of $\PGL_2(\Q_p)$. In the ramified 
situation, we need a more refined version of the Cerednik-Drinfeld theorem, since our points will belong to 
extension of a quadratic \emph{ramified} extension of $\Q_p$, for which the uniformization theorem does not hold. 
Therefore, we first need to describe $X/\Q_p$ over $\Q_p$ in terms of \emph{twisted Mumford curves} as 
in \cite[Theorem 4.3']{JL}, and then carefully describe the image of Heegner points in terms of these twisted objects. 

\begin{acknowledgements} 
We are grateful to A. Iovita for many interesting discussion about the topics of this paper.
M.L. is supported by PRIN 
\end{acknowledgements} 
\section{The anticyclotomic $p$-adic $L$-function of $\phi$ and $K$}\label{sec2}

We begin by setting some notation and recalling some definitions. Let 
\[\HH_p:=\PP_1(\C_p)-\PP_1(\Q_p)\]
be Drinfeld's upper half plane. The group $\PGL_2(\Q_p)$ acts on $\HH_p$ from the left by fractional linear transformations. Let $\T=\T_p$ be the Bruhat-Tits tree of $\PGL_2(\Q_p)$ and denote by $\vec{\E}(\T)$,
$\mathcal{E}(\T)$ and $\mathcal{V}(\T)$ the sets of oriented edges, edges and vertices of $\T$, respectively. 
For any $e\in\vec{\E}(\T)$ we denote by $V(e)\subset \HH_p$ the inverse image of $e$ under the 
canonical reduction map $r:\HH_p\rightarrow\mathcal{V}(\T)\cup\vec{\E}(\T)$. 

Let $B$ be the definite quaternion algebra of discriminant $N^-$. Since $p\nmid N^-$, $B_p:=B\otimes\Q_p$ is isomorphic to $\M_2(\Q_p)$. Fixing an isomorphism $\iota\colon B_p\rightarrow \M_2(\Q_p)$, we view the vertices of $\T$ as the maximal ($\Z$-)orders in $B_p$ and the edges as the Eichler orders of level $p$.
Let $R$ be an Eichler $\Z[1/p]$-order in $B$ of level $N^+$ (it is unique up to conjugation) and denote by $R_1^\times$ the group of the elements of reduced norm $1$ in $R$. Let $\Gamma :=\iota(R_1^\times)\subset \mathrm{PSL}_2(\Q_p)$ be the image of $R_1^\times$ under $\iota$.

Finally, let $\mathcal{P}_{k-2}$ be the $(k-1)$-dimensional $\C_p$-vector space of polynomials of degree at most $k-2$ with coefficients in $\C_p$. The linear group $\GL_2(\Q_p)$ acts on the right on $\mathcal{P}_{k-2}$ in the following way:
$$P(x)\cdot \beta:=\frac{(cx+d)^{k-2}}{(\mathrm{det}(\beta))^\frac{k-2}{2}}P\left(\frac{ax+b}{cx+d}\right),\qquad\beta=\begin{pmatrix}
a & b\\ c & d
\end{pmatrix},\quad P\in\mathcal{P}_{k-2}.$$
Since $\Q_p^\times\subset\mathcal{P}_{k-2}$ acts trivially, it is actually an action of $\PGL_2(\Q_p)$ on $\mathcal{P}_{k-2}$. This action also induces a left action on the dual space $\mathcal{P}_{k-2}^\vee:=\mathrm{Hom}(\mathcal{P}_{k-2},\C_p)$ given by
$$\beta\cdot\varphi(P):=\varphi(P\cdot\beta),\qquad\varphi\in\mathcal{P}_{k-2}^\vee.$$

We are now ready to introduce the following two definitions: 

\begin{definition}
A $p$-adic modular form of weight $k$ on $\Gamma$ 
is a $\C_p$-valued global rigid analytic function $f$ on $\HH_p$ satisfying
\begin{equation*}
f(\gamma z)=(cz+d)^k f(z)\qquad \text{for all}\ \gamma =\begin{pmatrix} a & b\\ c & d \end{pmatrix}\in\Gamma.
\end{equation*}
\end{definition}
\begin{definition}
A harmonic cocycle $c$ of weight $k$ on $\T$ is a $\mathcal{P}_{k-2}^\vee$-valued function on $\vec{\E}(\T)$ satisfying
$$c(e)=-c(\bar{e}),\qquad \sum_{\text{source}(e)=v}c(e)=0, \quad \forall v\in\mathcal{V}(\T),$$
where $\bar{e}$ is the unique edge obtained from $e$ by reversing the orientation.
\end{definition}
Denote by $C_\har(k)$ the $\C_p$-vector space of weight $k$ harmonic cocycles, and by $C_\har(k)^\Gamma$ the space of $\Gamma$-equinvariant harmonic cocycles of weight $k$, that is to say harmonic cocycles $c$ satisfying
$$c(\gamma e)=\gamma\cdot c(e)\qquad \forall  \gamma\in\Gamma.$$
For $c\in C_\har^\Gamma$, set
$$\langle c,c\rangle=\sum_{e\in\vec{\E}(\T)/\Gamma}w_e \langle c(e),c(e)\rangle,$$
where the sum is taken over a set of representatives for the $\Gamma$-orbits in $\vec{\E}(\T)$, the integer $w_e$ is the cardinality of the stabilizer of $e$ in $\Gamma$ and $\langle \hspace{0.1cm},\hspace{0.1cm} \rangle$ is the pairing on $\mathcal{P}_{k-2}^\vee$ defined in \cite[Sec. 1.2]{BDIS}.

\subsection{The $p$-adic modular form $f$ associated to $\phi$}
Denote $X=X_{N^+,pN^-}$ the Shimura curve (viewed as a scheme over $\Q$) 
attached to the quaternion algebra $\mathcal{B}=\mathcal{B}_{pN^-}$ 
of discriminant $pN^-$ and the Eichler order $\mathcal{R}=\mathcal{R}_{N^+}$ of level $N^+$ 
in $\mathcal{B}$. Let $R=\mathcal{R}[1/p]\subset B$ and $\Gamma=\iota(R_1^\times)$ be the image in $\mathrm{PSL}_2(\Q_p)$ of the elements of reduced norm $1$ in $R$. By the Cerednik-Drinfeld uniformization theorem, the quotient 
$\Gamma\backslash\HH_p$ is isomorphic to the set $X(\C_p)$ of $\C_p$-points of the Shimura curve $X$.

Let $\phi=\sum_{n=1}^\infty a_n q^n$ be a normalized eigenform of even weight $k$ on $\Gamma_0(N)$. Combining the Jacquet-Langlands correspondence and the Cerednik-Drinfeld uniformization theorem, we attach to $\phi$ a $p$-adic modular form $f$ of weight $k$ on $\Gamma$ such that
$$T_\ell(f)=a_\ell f\qquad \forall \ell\nmid pN^-,$$
where $T_\ell$ denotes the $\ell$-th Hecke correspondence on the Shimura curve $X$ defined above. 
This function is unique up to scaling by a nonzero element in $\C_p$.

\subsection{The harmonic cocycle $c_f$ associated to $f$}
We associate to $f$ a harmonic cocycle $c_f$ of weight $k$ defined by
$$c_f(e)(P)=\res_e(f(z)P(z)d(z)),\qquad P(z)\in\mathcal{P}_{k-2},$$
where $\res_e$ is the $p$-adic annular residue along $V(e)\subset \PP_1(\C_p)$. It can be shown that $c_f\in C_\har^\Gamma$ (see \cite[Lemma 3.13]{BD1}).

We normalize $f$ so that $\langle c_f,c_f\rangle=1$, it follows that $f$ is defined up to sign and that $c_f$ takes values in an extension at most quadratic of the field generated by the Fourier coefficients of $\phi$ (having fixed at the beginning an embedding of $\bar\Q$ in $\C_p$).

\subsection{The measure $\mu _f$ on $\PP_1(\Q_p)$ associated to $c_f$}\label{cocyclesec}
We associate now to the harmonic cocycle $c_f$ a measure $\mu _f$ on $\PP_1(\Q_p)$ in the following way.

First of all observe that $\PP_1(\Q_p)$ is identified with the space $\E_\infty (\T)$ of ends of $\T$. Recall that an end of $\T$ is an equivalence class of sequences $(e_n)_{n=1}^{\infty}$ of consecutive edges $e_n\in\vec{\E}(\T)$, where $(e_n)$ is identified with $(e_n')$ if and only if there exist $N$ and $N'$ with $e_{N+j}=e_{N'+j}'$ for all $j\geq 0$. After the identification of $\E_\infty(\T)$ with $\PP_1(\Q_p)$, the space $\E_\infty(\T)$ inherits a natural topology coming from the $p$-adic topology on $\PP_1(\Q_p)$. Each edge $e\in \vec{\E}(\T)$ corresponds to an open compact subset $U(e)\subset \E_\infty(\T)$ consisting of the ends having a representative containing $e$.

For all $e\in\vec{\E}(\T)$ and $P\in\mathcal{P}_{k-2}$ define
$$\mu_f(P\cdot\chi _{U(e)}):=\int_{U(e)}P(x)d\mu_f(x)=c_f(e)(P).$$
It extends to a functional on the space, which we denote by $\mathcal{A}_k$, of locally analytic $\C_p$-valued functions on $\PP_1(\Q_p)$ having a pole of order at most $k-2$ at $\infty$.

Extending the action of $\PGL_2(\Q_p)$ on $\mathcal{P}_{k-2}$ to $\mathcal{A}_k$, i.e. defining:
$$(\varphi*\beta)(x):=\frac{(cx+d)^{k-2}}{\mathrm{det}(\beta)^\frac{k-2}{2}}\varphi(\beta x),\qquad \beta\in\PGL_2(\Q_p),\ \varphi\in\mathcal{P}_{k-2},$$
one can easily see that $\mu_f$ satisfies the following properties:
\begin{lemma}
\begin{enumerate}
\item For any $P\in\mathcal{P}_{k-2}$, $\int_{\PP_1(\Q_p)}P(x)d\mu_f(x)=0$;\\
\item $\mu_f(\varphi*\gamma)=\mu_f(\varphi)$, for all $\gamma\in\Gamma$ and $\varphi\in\mathcal{A}_k$.
\end{enumerate}
\end{lemma}
\begin{proof}
\begin{enumerate}
\item Let $v$ be any vertex of $\T$. By identifying $\PP_1(\Q_p)$ with $\E_\infty(\T)$ one has
$$\int_{\PP_1(\Q_p)}P(x)d\mu_f(x)=\sum_{e,\ \mathrm{source}(e)=v}\int_{U(e)}P(x)d\mu_f(x)=\sum_{e,\ \mathrm{source}(e)=v}c_f(e)(P)=0,$$
where in the last equality we have used the harmonicity of $c_f$.\\
\item It suffices to check the formula on functions of the type $P\cdot\chi_{U(e)}$:
\begin{equation*}
\begin{split}
\mu_f((P\cdot\chi_{U(e)})*\gamma) & = \mu_f((P\cdot\gamma)(x)\cdot\chi_{U(e)}(\gamma x))=\mu_f((P\cdot\gamma)\cdot\chi_{U(\gamma^{-1}e)})\\
& = c_f(\gamma^{-1}e)(P\cdot\gamma)=\gamma^{-1}\cdot c_f(e)(P\cdot\gamma)=c_f(e)(P)=\mu_f(P\cdot\chi_{U(e)}).\\
\end{split}
\end{equation*}
\end{enumerate}
\end{proof}

\subsection{The measure $\mu_{f,\Psi ,\star}$ on $G$ associated to $\mu_f$ and $(\Psi,\star )$}
Now we want to pass from a measure on $\PP_1(\Q_p)$ to a measure on the Galois group of a suitable extension of $K$ related to the adjective anticyclotomic that appears in the title. Recall that our aim is to define a $p$-adic $L$-function associated to a pair $(\phi,K)$ where $\phi$ is an eigenform of even weight $k\geq2$ on $\Gamma_0(N)$ and $K$ is an imaginary quadratic field satisfying certain conditions. Up to now, we have used information coming only from $\phi$ (except for the factorization of $N$ which depends on $K$); from now on we will also take into consideration the field $K$.

Let $K_n$ denote the ring class field of $K$ of conductor $p^n$, and set $K_\infty=\bigcup_n K_n$. It is the maximal anticyclotomic extension which is unramified outside $p$. We have the following tower of extensions
$$\Q\subset K\subset H_p\subseteq K_0\subset K_1\subset \cdots\subset K_n\subset \cdots K_\infty,$$
where $K_0=H$ is clearly the Hilbert class field of $K$ and $H_p$ is the maximal subextension of $K_0$ over $K$ in which the prime of $K$ above $p$ splits completely. Note that if $p\cO_K=\p^2$ and $\p$ is a principal ideal then $H_p=K_0$.

We denote
$$G:=\Gal(K_\infty/H_p),\qquad \Delta :=\Gal(H_p/K),\qquad \tilde G:=\Gal(K_\infty/K).$$
We can interpret these groups in terms of ideles of $K$.
Denote by $\cO$ the $\Z[1/p]$-order $\cO_K[1/p]$ in $K$. For every rational prime $\ell$, set $K_\ell:=K\otimes\Q_\ell$, $\cO_\ell:=\cO\otimes\Z_\ell$ and $\hat\cO:=\cO\otimes\hat\Z$. Observe that $\hat\cO^\times$ is isomorphic to $\prod_\ell \cO_\ell^\times$, the product being taken over all primes $\ell$; in what follows it will also appear $\hat\cO ':=\prod_{\ell\neq p}\cO_{\ell}^\times$.

The adel ring of $K$ is $\hat K=\hat\cO\otimes\Q$. Note that, since we are assuming $p$ ramified in $K$, the $\Q_p$-algebra $K_\p$ is the completion of $K$ at the prime $\p$ above $p$ and it is a totally ramified extension of $\Q_p$ of degree $2$. 

If we denote by $\cO_{K,p^n}$ the order of $\cO_K$ of conductor $p^n$, then by definition
$$\Gal(K_n/K)=\Pic(\cO_{K,p^n})=\hat K^\times/K^\times\hat\cO_{K,p^n}^\times.$$
By Galois theory, one has
\begin{equation*}
\tilde G:=\Gal(K_\infty/K)=\varprojlim_n\Gal(K_n/K)=\hat K^\times/K^\times \bigcap_n \hat\cO_{K,p^n}^\times=\hat K^\times/K^\times\hat\cO '\Z_p^\times.
\end{equation*}
Observe that, by strong approximation, we can write
$$\hat\Q^\times=\Q^\times\prod_{\ell\neq p}\Z_\ell^\times\Q_p^\times,$$
therefore, since $\Q\subset K$ and $\prod_{\ell\neq p}\Z_\ell\subset \hat\cO '$, one has
$$\tilde G=\hat K^\times/K^\times \hat\cO'\hat \Q^\times.$$
Consider now the Galois group of $H_p$ over $K$. It is a quotient of $\Gal(K_0/K)=\hat K^\times/K^\times \hat\cO_K^\times$. In particular, since the prime of $K$ above $p$ splits completely in $H_p$, we have
\begin{equation*}
\Gal(H_p/K) =\hat K^\times/K^\times\hat\cO_K^\times K_\p^\times=\hat K^\times/K^\times\hat\cO^\times \Q_p^\times
\end{equation*}
and, again by strong approximation, one can write
$$\Delta :=\Gal(H_p/K)=\hat K^\times/K^\times\hat\cO^\times \hat\Q^\times$$
Finally, by Galois theory, $\Gal(K_\infty/H_p)$ is the unique subgroup $G$ of $\tilde G$ such that $\tilde G/G\simeq \Delta $, therefore
$$G:=\Gal(K_\infty/H_p)=K_\p^\times/\Q_p^\times \cO^\times.$$
For our purpose, we have to find a way to \textquotedblleft translate\textquotedblright the measure $\mu_f$ on $\PP_1(\Q_p)$ into a measure on $G=K_\p^\times/\Q_p^\times \cO^\times$. The trick is to embed $K_\p^\times/\Q_p^\times$ into a matrix algebra acting simply transitively on $\PP_1(\Q_p)$, obtaining then a bijection between $K_\p^\times/\Q_p^\times$ and $\PP_1(\Q_p)$. Finally a simple passage to a quotient will give the desired measure on $G$.

Equip both $\cO$ and $R$ with orientations, that is to say fix surjective homomorphisms
$$\mathfrak{o}\colon R\longrightarrow (\Z/N^+\Z)\times\prod_{\ell\mid N^-}\F_{\ell^2},$$
$$\alpha \colon \cO\longrightarrow (\Z/N^+\Z)\times\prod_{\ell\mid N^-}\F_{\ell^2}.$$
An embedding $\Psi \colon K\longrightarrow B$ is called an oriented optimal embedding of $\cO$ into $R$ if:
\begin{enumerate}
\item $\Psi(K)\cap R=\Psi(\cO)$, so that $\Psi$ induces an embedding of $\cO$ into $R$;
\item $\Psi$ is compatible with the chosen orientations on $\cO$ and $R$ in the sense that the following diagram commutes
\begin{diagram}[size=2em,nohug,labelstyle=\scriptstyle ]
\cO &      &\rTo^{\Psi} &       & R\\
    &\rdTo^\alpha &            & \ldTo^{\mathfrak{o}} & \\
    &      &(\Z/N^+\Z)\times\prod_{\ell\mid N^-}\F_{\ell^2}&  & \\
\end{diagram}
\end{enumerate}
Note that each embedding $\Psi$ corresponds to a vertex of $\T$. In fact $\Psi$ extends to an embedding of $K_\p$ in $B_p$ and, if $\cO_\p$ is the ring of the integers of $K_\p$, $\Psi(\cO_\p)$ is contained in a unique maximal order of $B_p$, which is a vertex of $\T$ by definition.

The group $R^\times$ acts by conjugation on the set of all oriented optimal embeddings of $\cO$ into $R$.
Write $\mathrm{emb}(\cO,R)$ for the set of all oriented optimal embedding of $\cO$ into $R$, taken modulo conjugation by $R_1^\times$.

The group $\Delta $ acts naturally on $\mathrm{emb}(\cO,R)$ in the following way. Let $\Psi\in\mathrm{emb}(\cO,R)$ and $\sigma \in \hat K^\times$ be a representative for $\bar \sigma \in \Delta=\hat K^\times/K^\times\hat\cO^\times \hat\Q^\times$, denote by $\hat\Psi$ the extension of $\Psi\colon K\rightarrow B$ to $\hat K\rightarrow \hat B$ and by $\hat R$ the adelisalition of $R$, that is $R\otimes\hat\Z$. Then $R_\sigma:=(\hat\Psi(\sigma)\hat R\hat\Psi(\sigma)^{-1})\cap B$ is an Eichler $\Z[1/p]$-order of level $N^+$, which inherits an orientation from the one of $R$. The element $\Psi\in\mathrm{emb}(\cO,R)$ is also an element of $\mathrm{emb}(\cO,R_\sigma)$; in fact $\Psi(\cO)\subset B$ and $\Psi(\cO)\subset(\hat\Psi(\sigma)\hat R\hat\Psi(\sigma)^{-1})$, so that $\Psi(\cO)\subset R_\sigma$, moreover $\Psi(K)\cap (\hat\Psi(\sigma)\hat R\hat\Psi(\sigma)^{-1})\subset \Psi(\cO)$ and $\Psi$ also respect the orientations on $\cO$ and $R_\sigma$. Since $B$ and $\Z[1/p]$ satisfy the Eichler condition, all the $\Z[1/p]$-order of a given level are conjugate in $B$, so that there exists $b\in B^\times$ such that $R=bR_\sigma b^{-1}$. Then we can define
$$\bar\sigma \Psi:=b\Psi b^{-1}\ \in\mathrm{emb}(\cO,R).$$
It is easy to see that the definition does not depend on the choice of the representative for $\bar \sigma$, so that it is really an action of $\Delta$. Moreover it is invariant under the action of $R_1^\times$ on $\Psi$, that is to say the action is well-defined on $\mathrm{emb}(\cO,R)$.

The following holds
\begin{proposition}
The group $\Delta$ acts freely on $\mathrm{emb}(\cO,R)$. The set $\mathrm{emb}(\cO,R)/\Delta$ has order $2$.
\end{proposition}
\begin{proof}
See \cite[Section 3]{Gross-Special-Values}.
\end{proof}
A pointed oriented optimal embedding of $\cO$ into $R$ is a pair $(\Psi,\star )$, where $\Psi$ is an oriented optimal embedding of $\cO$ into $R$ and $\star$ is a point of $\PP_1(\Q_p)$, which is not fixed by the action of $\Psi(K_\p^\times)$ by Moebius transformations. Since $R_1^\times$ acts on the oriented optimal embedding by conjugation and on $\PP_1(\Q_p)$ by Moebius transformations, we can consider the set of all pointed oriented optimal embeddings of $\cO$ into $R$ modulo $R_1^\times$, that we will denote by $\mathrm{emb}_\mathrm{p}(\cO,R)$. The elements of $\mathrm{emb}_\mathrm{p}(\cO,R)$ have an interpretation on the Bruhat-Tits tree: once fixed $\Psi\in\mathrm{emb}(\cO,R)$ which is a vertex $v$ on the quotient graph $\T/\Gamma$, to give $\star$ is equivalent to give an end of $\T/\Gamma$ originating from $v$.

The group $\tilde G$ acts on $\mathrm{emb}_\mathrm{p}(\cO,R)$. In particular, for each $n\geq 0$ the subgroup $G_n=\Gal(K_\infty/K_n)$ fixes $\Psi$ and the first $n+1$ edges of the end $\star$, so that $\sigma\in G_n$ sends $(\Psi,\star)$ to $(\Psi,\sigma\star)$ where $\sigma\star$ is an end which differs from $\star$ as from the ($n+2$)-th edge. While, as seen before, the quotient $\Delta$ moves the starting vertex $\Psi$.

As before, it holds
\begin{proposition}\label{action_tildeG}
The group $\tilde G$ acts freely on $\mathrm{emb}_\mathrm{p}(\cO,R)$. The set $\mathrm{emb}_\mathrm{p}(\cO,R)/\tilde G$ has order $2$.
\end{proposition}
\begin{proof}
See \cite[Lemma 2.13]{BDIS}.
\end{proof}
Fix a pointed oriented optimal embedding $(\Psi,\star)$. We want to define a measure $\mu_{f,\Psi,\star}$ on $G$.

Observe that, since $p$ is ramified in $K$, the completion $K_\p$ is a quadratic extension of $\Q_p$ so in particular a $\Q_p$-vector space of dimension $2$; it follows that the quotient of $K_\p^\times$ by the action of $\Q_p^\times$ is by definition $\PP_1(\Q_p)$. Since it is also a group, the action of $K_\p^\times/\Q_p^\times$ on itself is simply transitive. Therefore, fixing $\Psi$ we obtain a simply transitive action of $K_\p^\times/\Q_p^\times$ on $\PP_1(\Q_p)$, and the base point $\star$ determines a bijection \[\eta_{\Psi,\star}\colon K_\p^\times/\Q_p^\times\longrightarrow \PP_1(\Q_p)\]
by the rule $\eta_\Psi(x)=\iota\Psi(x^{-1})(\star)$. 
Note that $\iota\Psi(K_\p^\times)$ has two fixed points $z_\Psi,\bar z_\Psi\in K_\p^\times$ in $\HH_p$, which, up to a change of coordinates, are the roots of $X^2-D$.
We can identify $K_\p^\times/\Q_p^\times$ with $K_{\p,1}^\times$, the group of elements in $K_\p^\times$ of norm $1$, via the map sending $x\ (\mathrm{mod}\ \Q_p)$ to $x/\bar x$. If we choose $\star=\infty$ we can write the composition
\[\eta_{\Psi,\star}\colon K_{\p,1}^\times\overset{\sim}\longrightarrow K_\p^\times/\Q_p^\times\overset{\sim}\longrightarrow\PP_1(\Q_p)\]
in a simple explicit way (abusing notation we have denoted also by $\eta_{\Psi,\star}$ the composition of $\eta_{\Psi,\star}$ with $x\mapsto x/\bar x$).

We first look at the inverse $\eta_{\Psi,\infty}^{-1}\colon\PP_1(\Q_p)\longrightarrow K_{\p,1}^\times$.
Let $a\in\PP_1(\Q_p)$, it is the image of $\star=\infty$ under the Moebius transformation associated to $\iota\Psi(a-z_\Psi)=\iota\Psi\left(\left(\frac{1}{a-z_\Psi}\right)^{-1}\right)$.
So that, applying the identification $x\mapsto x/\bar x$, one obtains
\begin{equation}\label{corr}
\eta_{\Psi,\infty}^{-1}(a)=\frac{a-\bar z_\Psi}{a-z_\Psi}.
\end{equation}
Compute the inverse
$$\frac{a-\bar z_\Psi}{a-z_\Psi}=\alpha,\qquad \alpha\in K_{p,1}^\times\quad\Longrightarrow \quad a=\frac{\alpha z_\Psi-\bar z_\Psi}{\alpha-1}.$$
So that
$$\eta_{\Psi,\infty}(\alpha)=\frac{\alpha z_\Psi-\bar z_\Psi}{\alpha-1}.$$
Pullback by $\eta_{\Psi,\star}$ and $\eta_{\Psi,\star}^{-1}$ on functions preserve local analyticity and so we get a natural, continuous isomorphism
$$(\eta_{\Psi,\star}^{-1})^*\colon \mathcal{A}(K_{p,1}^\times)\longrightarrow \mathcal{A}(\PP_1(\Q_p))$$
between the $\Q_p$-algebra of locally analytic functions on $K_{p,1}^\times$ and the $\Q_p$-algebra of locally analytic functions on $\PP_1(\Q_p)$.

Let $u_\Psi:=\Psi(\sqrt{D})$, it is an element of reduced trace zero, as one can see looking at the characteristic polynomial of $\iota(u_\Psi)$, that is $X^2-D$. Define the polynomial $P_{u_\psi}\in\mathcal{P}_2$ as
$$P_{u_\Psi}(x)=\mathrm{trace}\left(\iota(u_\Psi)
\begin{pmatrix}
x\\ 1
\end{pmatrix}
\begin{pmatrix}
1 & -x\\
\end{pmatrix}
\right).$$ It is a polynomial with coeffcients in $\Q_p$ with the property that
\begin{equation}\label{invar}
P_{u_\Psi}\cdot\iota\Psi(\alpha)=P_{u_\Psi},\qquad\forall\alpha\in K_\p^\times.
\end{equation}
Now recall that in the previous section we have defined a distribution $\mu_f$ on the space of locally analytic functions on $\PP_1(\Q_p)$ having a pole of order at most $k-2$ in $\infty$, therefore, by restriction, also a linear functional on the free $\mathcal{A}(\PP_1(\Q_p))$-module of rank one $P_{u_\Psi}^\frac{k-2}{2}\cdot\mathcal{A}(\PP_1(\Q_p))\subset\mathcal{A}_k$, which is again a distribution (see \cite{Tei}, Proposition 9). 
We use it to define the distribution $\mu_{f,\Psi,\star}$ on $K_{p,1}^\times$ setting
$$\mu_{f,\Psi,\star}(\varphi ):=\mu_f(P_{u_\Psi}^\frac{k-2}{2}\cdot(\eta_{\Psi,\star}^{-1})^*(\varphi)),\qquad\text{for}\ \varphi\in\mathcal{A}(K_{\p,1}^\times).$$
Observe that under the identification $x\mapsto x/\bar x$ we can write $G=K_\p^\times/\Q_p^\times \cO^\times$ as
$$G=K_{p,1}^\times/\cO_1^\times,$$
where $\cO_1^\times$ is the group of norm one elements in $\cO^\times$.
By the theorem of $S$-units, $\cO^\times$ is a free abelian group of rank $1$, so that the quotient $\cO^\times/\Z[1/p]^\times$ is a finite group $\Z/n\Z$ for some $n$. In particular,
$$\left |\cO^\times/\Z[1/p]^\times\right | >1,$$
because, if we write $p$ as $xy$ with $x,y\in\p\subset \cO_K$ (recall that we assumed $p\cO_K=\p^2$) then $\frac{x}{p}\in \cO^\times$ but $\frac{x}{p}\notin\Z[1/p]^\times$. Consider the projection map
\begin{diagram}[l>=4em]
K_{p,1}^\times\simeq K_\p^\times/\Q_p^\times & \rOnto & K_\p^\times/\Q_p^\times\cO^\times\simeq K_{p,1}^\times/\cO_1^\times.
\end{diagram}
Its kernel is $\cO_1^\times\simeq\cO^\times /(\Q_p^\times\cap\cO^\times)=\cO^\times/\Z[1/p]^\times$, which is a finite group. In particular, every element of $\cO_1^\times$ is a root of unity, thus $\cO_1^\times=\{\pm 1\}$ since $K$ is an imaginary quadratic field, and as such it does not contain any cyclotomic field $\Q(\zeta_n)$ for $n\neq 2$.
This implies that the kernel of the projection $K_{p,1}^\times \twoheadrightarrow K_{p,1}^\times/\cO_1^\times=G$ has order $2$.

Now, to give a function $\varphi $ on $G=K_{p,1}^\times/\cO_1^\times $ is equivalent to give a function $\varphi$ on $K_{p,1}^\times$ such that $\varphi(x)=\varphi(ux)$
for all $x\in K_{p,1}^\times$ and $u\in\cO_1^\times$. Observe that $\mu_{f,\Psi,\star}$ is invariant under translation by $\cO_1^\times$. Indeed, if we denote by $\varphi\times u$ the function $x\longmapsto \varphi(ux)$ ($x\in K_{p,1}^\times$) then
$$\mu_{f,\Psi,\star}(\varphi\times u)=\mu_f((P_{u_\Psi}^\frac{k-2}{2}\cdot(\eta_{\Psi,\star}^{-1})^*(\varphi\times u))=\mu_f((P_{u_\Psi}^\frac{k-2}{2}\cdot(\varphi\times u)\circ \eta_{\Psi,\star}^{-1}),$$
but for $x\in\PP_1(\Q_p)$
$$(P_{u_\Psi}^\frac{k-2}{2}\cdot(\varphi\times u)\circ \eta_{\Psi,\star}^{-1})(x)=P_{u_\Psi}^\frac{k-2}{2}(x)\cdot\varphi(u \eta_{\Psi,\star}^{-1}(x))=P_{u_\Psi}^\frac{k-2}{2}(x)\cdot\varphi(\eta_{\Psi,\star}^{-1}(\gamma x))$$
where $\gamma\in\iota(\Psi(\cO_1^\times))\subset \Gamma$ is the matrix of $\Gamma$ which correspond to $u\in\cO_1^\times$. Using \eqref{invar} and that $c_f\in C_\mathrm{har}^\Gamma$ is $\Gamma$-invariant we obtain that
$$\mu_f(P_{u_\Psi}^\frac{k-2}{2}\cdot(\varphi\times u)\circ \eta_{\Psi,\star}^{-1})=\mu_f((P_{u_\Psi}^\frac{k-2}{2}\cdot(\eta_{\Psi,\star}^{-1})^*\varphi)*\gamma)=\mu_f(P_{u_\Psi}^\frac{k-2}{2}\cdot(\eta_{\Psi,\star}^{-1})^*\varphi)=\mu_{f,\Psi,\star}(\varphi).$$
The invariance of $\mu_{f,\Psi,\star}$ under translation by $\cO_1^\times$ implies that, if $\varphi\in\mathcal{A}(K_{p,1}^\times)$ is a locally analytic function on $K_{p,1}^\times$ and $V$ is a compact open subset of $K_{p,1}^\times$, then
$$\int_{uV}\varphi d\mu_{f,\Psi,\star}=\int_V (\varphi\times u) d\mu_{f,\Psi,\star},\quad \forall u\in\cO_1^\times$$
just for a change of variable. In particular if $\varphi\in\mathcal{A}(G)$ is a locally analytic function on $G=K_{p,1}^\times/\cO_1^\times$ and $\mathcal{F}$ is a fundamental domain for the action of $\cO_1^\times$ on $K_{p,1}^\times$ then
$$\int_{u\mathcal{F}}\varphi d\mu_{f,\Psi,\star}=\int_\mathcal{F}\varphi d\mu_{f,\Psi,\star}.$$
It follows that we can define a measure, denoted also by $\mu_{f,\Psi,\star}$, on $G$ setting
$$\mu_{f,\Psi,\star}(\varphi)=\int_G\varphi d\mu_{f,\Psi,\star}:=\int_\mathcal{F}\varphi d\mu_{f,\Psi,\star},\quad \text{for all}\ \varphi\in\mathcal{A}(G).$$

\subsection{Measure $\mu_{f,K}$ on $\tilde G$ associated to $\phi$ and $K$}
Extend $\mu_{f,\Psi,\star}$ to a $\C_p$-valued measure $\mu_{f,K}$ on $\tilde G$ by the rule
$$\mu_{f,K}(\delta U):=\mu_{f,\Psi^{\delta^{-1}},\star^{\delta^{-1}}}(U),\qquad U\subset G,\ \delta\in\tilde G.$$
For $\bar\delta\in\Delta$, choose a lift $\delta$ of $\bar\delta$ to $\tilde G$, so that $\tilde G$ is a disjoint union of $G$-cosets:
$$\tilde G=\bigcup_{\bar\delta\in\Delta}\delta G.$$
If $\varphi\in\mathcal{A}(\tilde G)$ is any locally analytic function on $\tilde G$, then
$$\int_{\tilde G}\varphi(x) d\mu_{f,K}(x)=\sum_{\bar\delta\in\Delta}\int_G \varphi(\delta x) d\mu_{f,\Psi^{\delta^{-1}},\star^{\delta^{-1}}}(x).$$
\begin{proposition}\label{dep_on_tildeG}
The distribution $\mu_{f,K}$ depends on the choice of $(\Psi,\star)$ only up to translation by an element of $\tilde G$, and up to sign. Its restriction to $G$ is $\mu_{f,\Psi,\star}$.
\end{proposition}
\begin{proof}
If $(\Psi,\star)$ is replaced by $(\gamma\Psi\gamma^{-1},\gamma\star)$ with $\gamma\in R_1^\times$, then the associated distribution is unchanged since $\mu_f$ is invariant under the action of $\Gamma$. By Proposition \ref{action_tildeG} there are only two $\tilde G$-orbits. If $(\Psi,\star)$ and $(\Psi',\star')=\alpha\cdot (\Psi,\star)$ are in the same $\tilde G$-orbits, the associated distributions differ by translation by $\alpha$. If $(\Psi,\star)$ and $(\Psi',\star')$ belong to different $\tilde G$-orbits, the associated distributions differ by translation by an element of $\tilde G$ and, eventually, by the sign.
\end{proof}

\subsection{The $p$-adic $L$-function associated to $\phi$ and $K$}

We are now ready to define the notion of $p$-adic $L$-function attached to $\phi$ and $K$.  
\begin{definition}Let 
\[\mathcal{X}=\Hom_\mathrm{cont}(\tilde{G},\bar\Q_p^\times)\] 
be the set of $\bar\Q_p^\times$-valued 
continuous characters $\chi:\tilde{G}\rightarrow\bar\Q_p^\times$. 
\begin{enumerate}
\item Let $(\Psi,\star)$ be a representative for a class in $\mathrm{emb}_\mathrm{p}(\cO,R)$. The 
\emph{partial anticyclotomic $p$-adic $L$-function} attached to $\phi$, $K$ and $(\Psi,\star)$ is the function 
$\chi\mapsto L_p(\phi,K,\Psi,\star,\chi)$ defined for $\chi\in\mathcal{X}$ as 
\[L_p(\phi,K,\Psi,\star,\chi):=\int_G \chi(g)d\mu_{f,\Psi,\star}(g).\]

\item The \emph{anticyclotomic $p$-adic $L$-function} attached to $\phi$ and $K$ is the function
$\chi\mapsto L_p(\phi,\chi)$ defined for $\chi\in\mathcal{X}$ by   
\[L_p(\phi,K,\chi):=\int_{\tilde G} \chi(g)d\mu_{f,K}(g).\]
\end{enumerate}
\end{definition}
Let $\mathrm{log}\colon \C_p^\times\longrightarrow \C_p$ be a branch of the $p$-adic logarithm. It gives a homomorphism $\mathrm{log}\colon K_\p^\times\longrightarrow K_\p$ which is $0$ on $\cO_1^\times$, and hence, by passing to the quotient, a homomorphism from $G$ to $K_\p$ which extends uniquely to a homomorphism\[\mathrm{log}\colon \tilde G\longrightarrow K_\p\] as $G$ has finite order in $\tilde G$.
For $g\in\tilde G$ and $s\in\Z_p$, define $$g^s:=\mathrm{exp}(s\cdot\mathrm{log}(g)),$$ where $\mathrm{exp}$ is the usual $p$-adic exponential.
For any finite order character $\chi:\tilde{G}\rightarrow\bar\Q_p$, let $\chi_s$ denote the 
character defined by $\chi_s(g)=\chi(g)\cdot g^s$. In this case, we put 
\[L_p(\phi,K,\chi,s)=L_p(\phi,K,\chi_{s-\frac{k}{2}})\quad\text{ and }\quad L_p(\phi,K,\Psi,\star,\chi,s)=L_p(\phi,K,\Psi,\star,\chi_{s-\frac{k}{2}}).\]
For $\chi$ a finite order character of $\tilde{G}$, we define $L_p'(\phi,K,\chi,s)$ and $L_p'(\phi,K,\Psi,\star,\chi,s)$ to be the $p$-adic derivative 
of the $p$-adic analytic functions $s\mapsto L_p(\phi,K,\chi,s)$ and 
$s\mapsto L_p(\phi,K,\Psi,\star,\chi,s)$, respectively. 


\subsection{Interpolation formulae} 
We now fix an embedding $\bar\Q_p\hookrightarrow \C$, thus allowing to consider $p$-adic 
numbers as complex numbers. 
Then we expect that the $p$-adic $L$-function satisfies the following interpolation formula: 
\begin{equation}\label{interp}
\left |\int_{\tilde G}\chi(g) d\mu_{f,K}\right | ^2=|L_p(\phi,K,\chi,k/2)|^2\doteq L^{\mathrm{alg}}(\phi,K,\chi,k/2)
\end{equation}
for any finite order character $\chi$ of $\tilde{G}$, where $\doteq$ indicates an equality up to an algebraic non-zero factor and $L^{\mathrm{alg}}(\phi,K,\chi,k/2)$ is a suitable normalisation of $L(E/K,1)$, obtained by dividing $L(\phi,K,k/2)$ by an appropriate complex period. For $k=2$, formula \eqref{interp} has been proved in \cite[Theorem 1.4.2]{YZZ}. 

\section{Twisted Cerednik-Drinfeld uniformization and Heegner points} 
The aim of this Section is to review a twisted version of the Cerednik-Drinfeld theorem which we will use for our purposes. 
As a general notation, if $T\rightarrow S$ and $X\rightarrow S$ are schemes over a base scheme $S$, we denote $X_T=X\times_ST$. If $F/\Q$ is a field extension, and $X$ is a scheme over $\Spec(\Q)$, 
we will write $X_{F}$ for $X_{\Spec(F)}$ and sometimes we will simply write $X$ when 
it is clear over which field the scheme is considered.  

\subsection{Twisted uniformization of Shimura curves}
From now on we will look at the $p$-adic $L$-function attached to the modular form of weight $2$ associated by modularity to an elliptic curve over $\Q$. In particular, let $E$ be an elliptic curve over $\Q$ of conductor $N$, $p$ be a prime such that $p\parallel N$ and $\phi$ be the normalized eigenform on $\Gamma_0(N)$ associated to $E$ by modularity. Using the construction of the previous section, we can define a $p$-adic $L$-function which interpolate the algebraic part of the special value $L(E/K,\chi,1)$ of the complex $L$-function of $E/K$ twisted by $\chi$, for all finite order ramified characters of $\tilde{G}$. We will denote it by $L_p(E/K,\chi,s)$ instead of $L_p(\phi/K,\chi,s)$ to recall that the modular form $\phi$ is that associated to $E/\Q$.

Recall the notation fixed in Section \ref{sec2}: $B$ is the quaternion algebra over 
$\Q$ of discriminant $N^-$; $R$ is a fixed Eichler $\Z[1/p]$-order in $B$ of level $N^+$; $R_1^\times$ is the group of the elements of reduced norm $1$ in $R$; $\iota:B_p=B\otimes_{\Q}\Q_p\simeq\M_2(\Q_p)$ is a fixed isomorphism; 
$\Gamma :=\iota(R_1^\times)\subset \mathrm{PSL}_2(\Q_p)$ is the image of $R_1^\times$ under $\iota$. We also define $\tilde\Gamma_0=\iota(R^\times)$ and $\Gamma_0=\iota(R_+^\times)$, where 
$R_+^\times$ is the subgroup of $R^\times$ consisting of elements with even valuation. Clearly, 
the image of $\Gamma_0$ in $\PGL_2(\Z_p)$ is $\Gamma$, and therefore $\Gamma_0\backslash\mathcal{H}_p=\Gamma\backslash\mathcal{H}_p$. 
Also, $\Gamma_0$ is normal in $\tilde{\Gamma}_0$
and we let $\{1,w_p\}$ be representatives of the quotient group 
$W=\tilde\Gamma_0/\Gamma_0$.   

Let $\hat{\mathcal{H}}_p$ be Drinfeld formal scheme over $\Spf(\Z_p)$ whose generic fiber is $\mathcal{H}_p$. 
The quotient $\Gamma_0\backslash\hat{\mathcal{H}}_p$ is then the formal completion along its closed fiber 
of a scheme $X_{\Gamma_0}$, which is projective over $\Z_p$. 

As in Section \ref{sec2}, denote $X=X_{N^+,pN^-}$ the Shimura curve (viewed as a scheme over $\Q$) 
attached to the quaternion algebra $\mathcal{B}=\mathcal{B}_{pN^-}$ 
of discriminant $pN^-$ and the Eichler order $\mathcal{R}=\mathcal{R}_{N^+}$ of level $N^+$ 
in $\mathcal{B}$. It is well known that over the quadratic 
unramified extension $\Q_{p^2}$ of $\Q_p$ we have an isomorphism of $\Q_{p^2}$-schemes 
$X_{\Gamma_0}\simeq X$. Moreover, let 
\[[\varphi] \in H^1\left(\Gal(\Q_{p^2}/\Q_p),\Aut(X_{\Q_{p^2}})\right)\]
be the class represented by the cocycle 
$\varphi:\Gal(\Q_{p^2}/\Q_p)\rightarrow\Aut(X_{\Q_{p^2}})$ 
which maps the generator $\Frob_p$ of $\Gal(\Q_{p^2}/\Q_p)$ to 
$w_p$.  Then (\emph{cf}. \cite[Theorem 4.3']{JL}) there is an isomorphism 
\[X_{\Gamma_0}^\varphi\overset \sim \longrightarrow X\] 
of schemes over $\Q_p$, where $X_{\Gamma_0}^\varphi$ denotes the twist of $X_{\Gamma_0}$ by $\varphi$.  
In particular, 
if $F$ is any extension of $\Q_p$ such that $F\cap\Q_{p^2}=\Q_p$ and $L=F\cdot\Q_{p^2}$, then we have 
\[X_{\Gamma_0}^\varphi(F)= X_{\Gamma_0}(L)^{\Frob_\p=w_p}\] 
where $\Frob_\p=\Frob_\p(L/F)$ denotes the generator of $\Gal(L/F)$ and 
\[X_{\Gamma_0}(L)^{\Frob_\p=w_p}=
\{P\in X_{\Gamma_0}(L)| P^{\Frob_\p}=w_p(P)\}.
\] 

\subsection{Twisted uniformization of elliptic curves} \label{sec:twisted}
As a consequence of the Jacquet-Langlands correspondence,  
it is known that there exists a modular uniformization 
\[\pi_E:\Div^0(X)\longepi E\] 
of $\Q_p$-schemes. Let 
\[\Phi_\mathrm{Tate}: \C_p^\times/q^\Z\longrightarrow E(\C_p)\] 
be the Tate uniformization of $E$, where $q$ is the Tate period, which is defined 
over $\Q_{p^2}$ if $E$ has non-split multiplicative reduction over $\Q_p$, 
and over $\Q_p$ if $E$ has split multiplicative reduction over $\Q_p$.  
Let $\log=\log_q$ be the choice of $p$-adic logarithm 
satisfying $\log_q(q)=0$, and define 
\[\log_E:E(\C_p)\longrightarrow \C_p\] 
by the formula $\log_E(P)=\log_q(x)$ for any $x\in\C_p$ such that $\Phi_\mathrm{Tate}\left([x])\right)=P$
(having denoted $[x]$ the class of $x$ in $\C_p^\times/q^\Z$). 

Over $\Q_{p^2}$, composing $\pi_E$ with the isomorphism $X_{\Gamma_0}\simeq X$, we get a map  
\begin{equation}\label{pi} \Div^0\left(X_{\Gamma_0}\right)\overset\sim
\longrightarrow \Div^0\left(X\right)\overset{\pi_E}\longrightarrow E.\end{equation}
This map can be described explicitly by means of $p$-adic Coleman integrals (\emph{cf}. \cite{Tei}). 
First, consider the map 
\[\mathcal{H}_p\times\mathcal{H}_p\longrightarrow \C_p^\times\] 
defined by 
\[[x]-[y]\longmapsto \mint_{[x]-[y]}\omega_f=\mint_{\mathbb{P}_1(\Q_p)}\frac{t-x}{t-y}d\mu_{f}\]
where $\int\!\!\!\!\!\times$ denotes the multiplicative 
integral defined by taking limits of Riemann products instead of sums (for more details see \cite{Dar}, Section 1) and recall that  
$\mu_f$ is the measure on $\mathbb{P}_1(\Q_p)$ constructed in \S \ref{cocyclesec}. 
Extending by $\Z$-linearity, we get a map 
\[\mint\omega_f:\Div^0(\mathcal{H}_p)\longrightarrow \C_p^\times.\] 
denoted $d\mapsto \int\!\!\!\!\!\times_d\omega_f$. 
The map in \eqref{pi} can be described as the map which takes a divisor of degree zero $d$ of $X_{\Gamma_0}$ 
defined over an extension of $\Q_{p^2}$ to 
\[\pi_{\Gamma_0,E}(d)=\Phi_\mathrm{Tate}\left(\mint_{\tilde d}\omega_f\right)\]
where $\tilde{d}$ is any degree zero divisor in $\Div^0(\mathcal{H}_p)$ 
which is sent to $d$ via the isomorphism of rigid analytic spaces 
$\Gamma_0\backslash\mathcal{H}_p\simeq X_{\Gamma_0}(\C_p)$. 


Over $\Q_p$, we have a map 
\begin{equation}
\label{unif}
\Div^0\left(X_{\Gamma_0}^\varphi\right)\overset\simeq\longrightarrow \Div(X) \longrightarrow E.
\end{equation}
Let $F$ be any finite field extension of $\Q_p$, let $L/F$ be the quadratic unramified extension of $F$, 
and let $q$ be as above the Tate period of $E$. Define
\[ \Sigma_F(E)= \begin{cases} F^\times/q^\Z,& \text{if $E/F$ has split multiplicative reduction}, \\
\{u\in L^\times/q^\Z \, | \, \mathrm{N}_{L/F}(u)\in q^\Z/q^{2\Z}\},& \text{otherwise},
\end{cases}
\] where $\mathrm{N}_{L/F}$ denotes the norm map. 
Tate's theory of $p$-adic uniformization implies then that the following diagram is commutative:
\[\xymatrix{
&\Sigma_F(E)\ar[r]^{\Phi_\mathrm{Tate}}_\simeq & E(F)\ar@{^(->}[r] & E(L) \\
\Div^0\left(X_{\Gamma_0}^\varphi(F)\right)=\hspace{-1cm}& \Div^0\left(X_{\Gamma_0}(L)^{\Frob_\p=w_p}\right)\ar[u]^{{\times \hskip -0.7em \int}\omega_f}\ar@{^(->}[rr] & & \Div^0\left(X_{\Gamma_0}(L)\right) \ar[u]^\simeq _{\pi_{\Gamma_0,E}}
}
\] where $\Phi_\mathrm{Tate}:L^\times/q^\Z\rightarrow E(L)$ denotes as above the Tate uniformization.

\begin{remark}
One can directly check that the dotted arrows factors through $\Sigma_F(E)$. 
If a divisor $d=\sum_ix_i$ in $\Div^0(X_{\Gamma_0}(L))$ satisfies  
$\Frob_\p(x_i)=w_p(x_i)$ for all $i$, 
using that the isomorphism $\Gamma_0\backslash\mathcal{H}_p\simeq X_{\Gamma_0}(\C_p)$ 
is equivariant for the action of $w_p$, and the action of $w_p$ on $\mu_f$ is via $a_p$, we see that
\[\mathrm{N}_{L/F}\left(\mint_{\tilde d}\omega_f\right)=
\left(\mint_{\tilde d}\omega_f\right)\cdot \left(\mint_{w_p(\tilde d)}\omega_f\right)=\left(\mint_{\tilde d}\omega_f\right)^{1+a_p}\] which shows that $\int\!\!\!\!\!\times_{\tilde d}\omega_f$ belongs to $F^\times$ if $a_p=1$ and 
is trivial if $a_p=-1$.  
\end{remark}

\subsection{Heegner points}\label{sec:Heegner} Let $P_H\in X(H)$ be a Heegner point of conductor $1$, 
defined over the Hilbert class field of $K$. Let $y_H=\pi_E(P_H)$. 
Our next task is to use the twisted version of the Cerednik-Drinfeld theorem 
recalled above to describe the localization of the point $y_H$.  

Recall that $\p=\p_K$ denotes the unique prime of $K$ above $p$. 
Fix a prime $\p_H$ of $H$ above $p$, and let $\p_{H_p}$ be its restriction to $H_p$. 
The completion $H_{p,\p_{H_p}}$ of $H_p$ at $\p_{H_p}$ is then isomorphic to the 
completion $K_\p$ of $K$ at $\p$, and the completion $H_{\p_H}$ of $H$ at $\p_H$ is 
a (finite unramified) extension of $K_\p$.  
\footnote{Denote $H_\p=H_{\p_H}$ to simplify the notation 
(understanding then the choice of $\p_H$ which has been made).} We also observe that, 
since $p$ splits completely in $H_{p}$, the prime $\p_H$ is the unique prime of $H$ above $\p_{H_p}$.

We still denote $P_H$ both the point in $X(H_{\p_H})$ obtained by localization, 
and the point  in $X_{\Gamma_0}^\varphi(H_{\p_H})$ corresponding to $P_H$
under the isomorphism $X_{\Gamma_0}^\varphi\simeq X$. If $F=H_{\p_H}\cdot \Q_{p^2}$, then the point 
$P_H\in X_{\Gamma_0}^\varphi(H_{\p_H})$ 
corresponds to a point $\tilde{P}_H\in X_{\Gamma_0}(F)$ on which the generator 
$\Frob_p$ of $\Gal(F/H_{\p_H})$ acts by $w_p$. Finally denote $\xi(P_H)\in\mathcal{H}_p$ a representative of the point in $\Gamma_0\backslash\mathcal{H}_p$ 
corresponding to $\tilde{P}_H$ via the isomorphism
$\Gamma_0\backslash\mathcal{H}_p\simeq X_{\Gamma_0}(\C_p)$
of algebraic varieties over $\C_p$. 

Let $(A,i,C)$ be the modular point associated with $P_H$ by the modular interpretation of $X$, 
as described for example in \cite[Section 4]{BD-Heegner}. In particular, $A$ is an abelian 
variety defined over $H$, with complex multiplication by $\mathcal{O}_K$. 
Let $A^\mathrm{red}$ be the reduction of (an integral model of) $A$ modulo $\p_H$, which is then an abelian variety defined 
over the residue field of $H_{\p_H}$. It is known (see for example \cite[Proposition 5.2]{Molina}) that the reduction of 
endomorphism gives a map \[\Psi=\Psi(P_H):\End(A)\longrightarrow\End(A^\mathrm{red})\] which is an optimal 
embedding of $\mathcal{O}_K$ into an Eichler $\Z$-order $R_0$ of level $N^+p$ in the quaternion algebra $B$. 
Since any two Eichler $\Z[1/p]$-orders of the same level in $B$ are conjugate, we may assume that $R_0[1/p]=R$, where $R$ is the $\Z[1/p]$-order used to define the $p$-adic $L$-function. 
So, tensoring $\Psi$ with $\Z[1/p]$, we get an optimal embedding $\Psi:\mathcal{O}\rightarrow R$. 
Since $A$ has complex multiplication by $\mathcal{O}_K$, the point $\xi(P_H)\in\mathcal{H}_p$ 
representing the class of $\tilde{P}_H$ as above 
is fixed by the action of 
$\Psi(\mathcal{O}_K)$ and therefore coincides with one of the two fixed points for the action of $\Psi(\mathcal{O}_K)$ on $\mathcal{H}_p$. We normalize the choice of the fixed point so that the action of $K_\p^\times$
on the tangent line of $\mathcal{H}_p$ at $\xi(P_H)$ is via the character 
$z\mapsto z/\bar{z}$, as opposite to the character $z\mapsto {\bar{z}}/z$. 
We may thus write, with the usual notation, $\xi(P_H)=z_{\Psi}$; if we  
need to specify the point $P_H$ that we started with to construct $z_{\Psi}$, we write $z_{{\Psi}(P_H)}$.

If $\sigma\in\Gal(H/K)$, then we may consider the conjugate point $P_H^\sigma$ in $X(H)$, 
and its localization $P_H^\sigma$ in $X(H_{\p_H})$ as above. 
We then get a point $z_{\Psi(P_H^\sigma)}$ in $\mathcal{H}_p$; 
since $(\Psi(P_H))^\sigma=\Psi(P_H^\sigma)$, we have $ z_{\Psi(P_H^\sigma)}=
z_{(\Psi(P_H)) ^\sigma}$. Having fixed our point $P_H$ and $\Psi=\Psi(P_H)$, the 
last formula simply reads as $z_{\Psi(P_H^\sigma)}=z_{\Psi^\sigma}$. 
Since $\Gal(H/H_p)$ acts trivially on $\mathrm{emb}(\mathcal{O},R)$, and $\Gal(H_p/K)$ 
acts without fixed points, we see that there are 
exactly $[H:H_p]$ points whose associated point in $\mathcal{H}_p$ is 
a given $z_{{\Psi}^\sigma}$.

Observe that, since $H_{\p_H}/K_\p$ is unramified and $K_\p/\Q_p$ is totally ramified, the Galois group $\Gal(H_{\p_H}/\Q_p)$ can be written as 
\[\Gal(H_{\p_H}/K_\p)\times\Gal(K_\p/\Q_p)\simeq\langle\mathrm{Frob}_\p\rangle\times\langle\tau\rangle,\] where $\tau$ is the map induced on the completion $K_\p$ by the complex conjugation on $K$.
We can see the abelian variety $A$ as defined over $H_{\p_H}$, thus we can consider the map $f\colon A\rightarrow \Spec(H_{\p_H})$. We also have the endomorphism $\tau^\#$ induced on $\Spec(H_{\p_H})$ by the complex conjugation $\tau$. Let $\bar{A}$ be the fibred product of $A$ and $\Spec(H_{\p_H})$ over $\Spec(H_{\p_H})$ with respect to $f$ and $\tau^\#$, that is $\bar{A}=A\times_{H_{\p_H},\tau^\#} H_{\p_H}$. Denote by $p_1$ and $p_2$ the projections of $\bar{A}$ onto $A$ and $\Spec(H_{\p_H})$ respectively. For every $\varphi\in \End (A)$, we have $f\varphi p_1=\tau^\# p_2$
, therefore by the universal property of the fibred product there exists a unique $\rho\in \End(\bar{A})$ such that $\varphi p_1=p_1\rho$. Thus we obtain a map $\alpha\colon\End(A)\rightarrow\End(\bar{A})$. Now observe that $A^\mathrm{red}$ is isomorphic to $\bar{A}^\mathrm{red}$ as schemes, because $p_1$ is induced by $\tau$ and, since $\tau$ is the generator of a totally ramified extension, at level of residue fields it is the identity.
It follows that, when we pass to the reduction, the map $\bar{\Psi}\colon\End(\bar{A})\rightarrow\End(\bar{A}^\mathrm{red})$ satisfies $\bar{\Psi}\alpha=\Psi$. In particular, if $z\in \cO_K\subseteq\End(A)$ then $\alpha(z)=\tau(z)$ and $\bar{\Psi}(z)=\Psi(\tau(z))$. It follows that the optimal embedding $\bar{\Psi}\colon \cO\rightarrow R$ induced by $\bar{A}$ is given composing $\tau\colon \cO\rightarrow\cO$ with $\Psi\colon \cO\rightarrow R$, therefore if $\xi(P_H)\in\mathcal{H}_p$ is the point corresponding to the abelian variety $A$ then 
$\overline{\xi(P_H)}=\tau(\xi(P_H))$ is the point corresponding to $\bar{A}$. But $\bar{A}$ is just the modular point associated to $\bar{P}_H\in X(H)$, where $\bar{P}_H$ is the complex conjugate of $P_H\in X(H)$. 
In light of the above normalization, we therefore have
\begin{equation}\label{coleman}\pi_{E}(P_H-\bar{P}_H)= \pi_{\Gamma_0,E}\left(\xi(P_H)-\overline{\xi(P_H)}\right)=
\Phi_\mathrm{Tate}\left(\mint_{\bar{z}_\Psi}^{z_\Psi}\omega_f\right)=\Phi_{\mathrm{Tate}}\left(\mint_{\mathbb{P}^1_{\Q_p}}\frac{t-z_\Psi}{t-\bar{z}_\Psi}d\mu_{f}\right).\end{equation}

Fix a Heegner point $P_H\in X(H)$,
the optimal embedding $\Psi=\Psi(P_H)\in\mathrm{emb}(\mathcal{O},R)$ 
and the $p$-adic point $z_\Psi\in\mathcal{H}_p$ as above. 
For any character $\chi:\tilde{G}\rightarrow\bar\Q_p^\times$, denote $\Z[\chi]$ the 
ring extension of $\Z$ generated by the values of $\chi$. For a character $\chi$ factoring through $\Gal(H/K)$, 
define 
\[P_\chi=\sum_{\sigma\in\Gal(H/K)}P_H^\sigma\otimes\chi^{-1}(\sigma)\]
which is an element in the group ring $\Div(X(H))\otimes_\Z\Z[\chi]$.
Fix now a character $\chi:\tilde{G}\rightarrow \bar\Q_p$ which factors through 
$\Delta=\Gal(H_p/K)$ (so, it is trivial on $G$). 
Define the divisor \[P_{H_p}=\sum_{\sigma\in\Gal(H/H_p)}P_{H}^\sigma.\]
We then have in the group ring $\Div(X(H))\otimes_{\Z}\Z[\chi]$, 
\begin{equation}\label{eq7}
P_\chi=\sum_{\sigma\in\Delta}P_{H_p}^{\tilde\sigma}\otimes \chi^{-1}(\tilde\sigma)\end{equation}
where $\tilde\sigma\in\Gal(H/K)$ is any lift of $\sigma\in\Delta$.   

Following the notation in \S\ref{sec:twisted}, 
define \[z_\chi=\sum_{\sigma\in\Gal(H/K)}z_{\Psi^{\sigma}}\otimes\chi^{-1}(\sigma).\]
Since, as observed above, 
$\Gal(H/H_p)$ acts trivially on $\Psi$, we have 
\[\sum_{\sigma\in\Gal(H/H_p)}z_{\Psi^\sigma}=[H:H_p]\cdot z_\Psi\] 
as elements in $\Div(\mathcal{H}_p)$ 
and therefore 
\[z_\chi=[H_p:H]\cdot\sum_{\sigma\in\Delta}z_{\Psi^{\tilde\sigma}}\otimes\chi^{-1}(\tilde\sigma).\]
where as above $\tilde\sigma$ is any lift of $\sigma$. 
Similarly, we may consider the point $\bar{z}_\chi$ obtained with the same process, 
but starting with the point $\bar{P}_H$ instead of $P_H$. 
Using \eqref{coleman}, \eqref{eq7} and the fact that all points in the 
divisor $P_{H_p}^{\tilde\sigma}$ reduce to the same $z_{\Psi^{\tilde\sigma}}$, we have
\begin{equation}\label{coleman1}\pi_{E}(P_\chi-\bar{P}_\chi)=
\Phi_\mathrm{Tate}\left(\mint_{\bar{z}_\chi}^{z_\chi}\omega_f\right)=
[H:H_p]\cdot\Phi_{\mathrm{Tate}}\left(\prod_{\sigma\in\Delta}
\left(
\mint_{\mathbb{P}^1_{\Q_p}}
\frac{t-z_{\Psi^{\sigma}}}{t-\bar{z}_{\Psi^{\sigma}}}d\mu_{f}
\right)^{\chi^{-1}(\tilde\sigma)}
\right).
\end{equation}

\section{Zeros of the $p$-adic $L$-function at the central point $s=1$}

\subsection{Signs} 


Let $G_n=\Gal(K_n/K)$ be the Galois group of the ring class field of $K$ of conductor $p^n$ over $K$. Fix a character $\chi:G_n\rightarrow\bar{\Q}^\times$. We are interested in computing the sign of the functional equation of $L(E/K,\chi,s)$ at its center of symmetry $s=1$, which we denote $w(E,\chi)$. We will say that we have a \emph{change of sign phenomenon} if the sign of functional equation of $L(E/K,\chi,s)$ 
is different for $\chi=\mathbf{1}$ and $\chi$ ramified at $p$.

If we let $w_\ell(E,\chi)$ denote the local sign of the $L$-function of $E/K$ and $\chi$ at a rational prime $\ell$, then it is known 
that the sign $w(E,\chi)$ of $L(E/K,\chi,s)$ is 
\[w(E,\chi)=(-1)^{\#\Sigma(E,\chi)+1}\]
where $\Sigma(E,\chi)$ is the set of rational primes $\ell$ such that $w_\ell(E,\chi)\neq \eta_K(-1)$. 
See \cite[Theorem 1.3.2]{YZZ} or \cite[Sec. 1.1]{CV} for details. 

By \cite[(1.1.3)]{Zhang-AJM}, if $\ell\neq p$, then $\ell\in\Sigma(E,\chi)$ if and only if $\ell\mid N^-$. So let us consider the case of $p$.
Again by \cite[(1.1.3)]{Zhang-AJM}, $p\in \Sigma(E,\chi)$ if and only if the $p$-th Fourier coefficient of $\phi$, that is $a_p$, is the inverse of the $p$-th Fourier coefficient 
\[b_p=\sum_{\substack{\mathfrak{a}\subset \cO_K\\N_K(\mathfrak{a})=p}}\chi(\mathfrak{a})\]
of the Hecke theta series associated to $\chi$.
On the other hand, by a result of Cornut and Vatsal (see \cite[Lemma 1.1]{CV}), if $\chi$ is sufficiently ramified at $p$, $p$ does not belong to $\Sigma(E,\chi)$. In our situation, one can show 
that $p\not\in\Sigma(E,\chi)$ whenever $\chi$ is ramified at $p$. 

\begin{remark} 
If $\chi=\mathbf{1}$ is the trivial character, then $b_p=1$ because, since $p$ is ramified in $K$, there is only one prime in $\cO_K$ with norm $p$, that is the unique prime above $p$. This implies that $p\in\Sigma(E,\chi)$ if and only if $a_p=1$, i.e. $E$ has \emph{split} multiplicative reduction in $p$.
Therefore we may conclude that if $E$ has split multiplicative reduction, then we have at least one change of sign phenomenon. The same holds for all characters $\chi$ such that $\chi(\p)=1$. 
On the contrary, if $a_p=-1$, i.e. $E$ has \emph{non-split} multiplicative reduction in $p$, we have a change of sign phenomenon for all characters $\chi$ such that $b_p=\chi(\p)=-1$.
\end{remark}


We denote $\mathcal{Z}$ the subset of $\mathcal{X}$ consisting of finite order characters $\chi\colon\tilde{G}\rightarrow\bar{\Q}_p$ factoring through $\Delta=\Gal(H_p/K)$.
 




By definition, the value of $L_p(E/K,\Psi,\star,s)$ in $s=1$ does not depend on the choice of $\star$, i.e.
$$L_p(E/K,\Psi,\star,1)=L_p(E/K,\Psi,\star',1).$$
Therefore we can write $L_p(E/K,\Psi,1)$ instead of $L_p(E/K,\Psi,\star,1)$.

Define $h_p=[H_p:K]=|\Delta|$. Fix $\Psi=\Psi_1$ in $\mathrm{emb}(\cO,R)$, and let $\Psi_2,\dots,\Psi_{h_p}$ 
be its conjugates under the action of $\Delta$ on $\mathrm{emb}(\cO,R)$. We order the elements of $\Delta=\{\sigma_1,\dots,\sigma_{h_p}\}$ in such a way that $\Psi^{\sigma_i^{-1}}=\Psi_i$.

\begin{lemma}If $\chi\in\mathcal{Z}$, then 
$L_p(E/K,\chi,1)=0$ and in fact $L_p(E/K,\Psi,1)=0$ for all $\Psi$ in $\mathrm{emb}(\cO,R)$.
\end{lemma}
\begin{proof}
Observe that the first statement is trivial if $\chi$ is such that $a_p\chi(\p)=1$: in fact from the previous section, the set $\Sigma(E,\chi)$ contains $p$ and therefore it has even cardinality. It follows that $L(E/K,\chi,1)=0$ and, by interpolation, also $L_p(E/K,\chi,1)=0$.
We now show the stronger statement $L_p(E/K,\Psi,1)=0$ for all $\Psi$ in $\mathrm{emb}(\cO,R)$. This will imply the first one since
$$L_p(E/K,\chi,1)=\int_{\tilde{G}}\chi(g)d\mu_{f,K}(g)=\sum_{i=1}^{h_p}\int_G \chi(\sigma_i)d\mu_{f,\Psi_i,
\star_i}(g)
=\sum_{i=1}^{h_p}\chi(\sigma_i)L_p(E/K,\Psi_i,1),$$
where in the second equality we used the fact that $\chi$ factors through $\Delta$.

Fix $\Psi\in\mathrm{emb}(\cO,R)$; we now compute $L_p(E/K,\Psi,1)$. Recall that
$$L_p(E/K,\Psi,1)=\int_G d\mu_{f,\Psi,\star}=\int_\mathcal{F} d\mu_{f,\Psi,\star},$$
where $\mathcal{F}$ is a fundamental domain for the action of $\cO_1^\times$ on $K_{p,1}^\times$.

Let $v$ be the vertex of $\T$ corresponding to the embedding $\Psi$, $u_0$ be a generator for the finite group $\cO^\times/\Z[1/p]^\times$ and $u=u_0/\bar u_0$ the image of $u_0$ in $K_{p,1}^\times$. Denote $\gamma=\iota\Psi(u)\in\Gamma$, it has order $2$ since $|\cO^\times/\Z[1/p]^\times|=2$.

To know a fundamental domain for the action of $u$ on $K_{p,1}^\times$ is equivalent to know a fundamental domain for the action of $\gamma$ on $\PP_1(\Q_p)=\mathcal{E}_\infty(\T)$.
Consider the path joining $v$ and $\gamma v$; the ends rising from $v$ are equivalent to those arising from $\gamma v$ under the action of $\gamma$. Therefore the ends arising from the vertices of the path between $v$ and $\gamma v$ are a fundamental domain for the action of $\gamma$ on $\PP_1(\Q_p)$. More explicitly, if $\mathcal{E}'$ denote the set of oriented edges of $\T$ with the same source as the edges on the path from $v$ to $\gamma v$, but not containing these edges, then
$$\mathcal{F'}=\bigcup_{e\in \mathcal{E}'}U(e)$$
is a fundamental domain for the action of $\gamma$ on $\PP_1(\Q_p)$. Therefore $\eta_{\Psi,\star}^{-1}(\mathcal{F'})$ is fundamental domain for the action of $u$ on $K_{p,1}^\times$. It follows that
$$\int_\mathcal{F} d\mu_{f,\Psi,\star}=\int_\mathcal{F'} d\mu_f=\sum_{e\in \mathcal E'}\mu_f(\chi_{U(e)})=\sum_{e\in \mathcal E'} c_f(e).$$
Let $\vec{e}_0,\dots,\vec{e}_k$ be the edges of the path between $v$ and $\gamma v$, then, since $c_f$ is a cocycle, we have
$$\sum_{e\in \mathcal E'} c_f(e)=-c_f(\vec{e}_0)+(c_f(\vec{e}_0)-c_f(\vec{e}_1))+\dots+(c_f(\vec{e}_{k-1})-c_f(\vec{e}_k))+c_f(\vec{e}_k)=0$$ which concludes the proof. 
\end{proof}
When $L(E/K,\chi,1)=0$ (which happens for instance when $a_p\chi(\p)=1$) one expects that the annihilation of $L_p(E/K,\chi,1)$ is due to the existence of an effective non torsion rational point $P_K$ on the elliptic curve $E$. 
More precisely, we will relate $L_p'(E/K,\chi,1)$ to the existence of a global point in $E(\bar\Q)$ in the 
following, in the spirit of the main result of \cite[Corollary 7.2]{BD-CD}.

Let $z_{\Psi_i}$ and $\bar z_{\Psi_i}\in\HH_p$ be the two fixed points for 
$\iota\Psi_i(K_\p^\times)$ acting on $\HH_p$.
To simplify the notation, 
put 
\[L_p(E/K,\Psi_i,\chi,1)=L_p(E/K,\Psi_i,\infty,\chi,1).\]

\begin{lemma}\label{integr}
Let $\chi\in\mathcal{Z}$. Then the following equalities hold up to sign:  
\begin{enumerate}
\item $L_p'(E/K,\Psi_i,\chi,1)=2\int_{\bar z_{\Psi_i}}^{z_{\Psi_i}} f(z)dz,$ 
\item\label{firstder} $L_p'(E/K,\chi,1)=2\sum_{i=1}^{h_p}\chi(\sigma_i)\int_{\bar z_{\Psi_i}}^{z_{\Psi_i}} f(z)dz.$
\end{enumerate}
\end{lemma}
\begin{proof} The indeterminacy of the sign in these equations is unavoidable, since the harmonic 
cocycle attached to $f$ is only well defined up to sign. To prove the first statement, 
recall that
$$L_p(E/K,\Psi,\chi,s)=\int_G \chi(g)g^{s-1}d\mu_{f,\Psi,\star}(g),$$
therefore, since $\chi$ is trivial on $G$, we have 
\[L_p'(E/K,\Psi,\chi,s)|_{s=1}=\int_G(g^{s-1})'|_{s=1}d\mu_{f,\Psi,\star}(g)=\int_G\mathrm{log}(g) d\mu_{f,\Psi,\star}(g).\]
On the other hand, by definition of the Coleman $p$-adic line integral associated to the choice of the $p$-adic logarithm $\mathrm{log}$, one has
\[\int_{\bar z_\Psi}^{z_\Psi} f(z)dz=\int_{\PP_1(\Q_p)}\mathrm{log}\left(\frac{a-z_\Psi}{a-\bar z_\Psi}\right)d\mu_f(a)=2\int_G \mathrm{log}(g) d\mu_{f,\Psi,\infty}(g),\]
where in the last equality we used \eqref{corr} and $|\mathrm{Ker}(K_{p,1}^\times\twoheadrightarrow G)|=2$, and the first equality comes from 
the work of Teitelbaum \cite[Sec. 1]{Tei}. The second formula follows immediately from the first 
by integration, since comparing with the definition of the $p$-adic $L$-function we see that 
\[L_p(E/K,\chi,1)=\sum_{i=1}^{h_p} \chi(\sigma_i)L_p(E/K,\Psi_i,\chi,1)\]
because $\chi$ factors through $\Gal(H_p/K)$. 
\end{proof}

Let $P_H\in X(H)$ be a Heegner point of conductor $1$ and recall the point 
$P_\chi\in X(H)\otimes_{\Z}\Z[\chi]$ defined in \S \ref{sec:Heegner}. 
As in \S \ref{sec:Heegner}, denote $y_H\in E(H_\p)$ the 
point obtained as the image of $P_H$ via the modular uniformization $\pi_E$ and then localizing from $E(H)$ to $E(H_\p)$; we also denote $y_\chi$ the point in $E(H_\p)\otimes_{\Z}\Z[\chi]$ 
corresponding to $P_\chi$. The following is the main result of this paper: 

\begin{theorem}For all $\chi\in\mathcal{Z}$ we have 
$L_p'(E/K,\chi,1)=\frac{2}{[H:H_p]}\cdot\mathrm{log}_E(y_\chi/\bar y_\chi)$ up to signs. 
\end{theorem}

\begin{proof}
Combine \eqref{coleman1} and Lemma \ref{integr}, \eqref{firstder}. 
\end{proof}

\bibliographystyle{amsalpha}
\bibliography{paper}
\end{document}